\documentclass[11pt, numbers=endperiod, abstract=true]{scrartcl}

\usepackage{indentfirst}
\usepackage{geometry}
\usepackage[hypertexnames=false,colorlinks,linkcolor={blue},citecolor={blue},urlcolor={blue}]{hyperref}
\usepackage[numbers,sort&compress]{natbib}

\usepackage{amssymb,amsmath,mathrsfs}
\usepackage{amsthm,stmaryrd}
\usepackage{eucal}
\usepackage{empheq}
\usepackage{graphicx,xcolor,float,caption}
\usepackage{epsfig,subcaption,empheq}
\usepackage{enumitem}
\setlist{
	leftmargin=*, topsep = 4pt, 
	itemsep = 1pt, 
}

\usepackage{gmacro}

\title{KKT Reformulations for Single Leader and Multi-Follower Games}
\def\runningtitle{KKT Reformulations for Single Leader\dots}

\author[1,2]{Parin Chaipunya}
\author[3]{Thirumulanathan D}
\author[3]{Joydeep Dutta}

\affil[1]{
	Department of Mathematics, Faculty of Science, King Mongkut's University of Technology Thonburi\\
	126 Pracha Uthit Rd., Bang Mod, Thung Khru, Bangkok 10140, Thailand.
}
\affil[2]{
	The Joint Graduate School of Energy and Environment (JGSEE), King Mongkut's University of Technology Thonburi\\
	126 Pracha Uthit Rd., Bang Mod, Thung Khru, Bangkok 10140, Thailand.
}
\affil[3]{
	Indian Institute of Technology Kanpur (IIT Kanpur), \\
  Kalyanpur, Kanpur, Uttar Pradesh 208016, India.
}

\date{}

\usepackage{hfset}

\begin{document}

\maketitle \vspace{-4.5em}
\begin{abstract}
	
  We consider a bilevel optimization problem having a single leader and multiple followers. 
  The followers choose their strategies simultaneously, and are assumed to converge to a Nash equilibrium strategy profile. 
  We begin by providing a practical example of such a problem in an oligopoly setting. 
  We then show the existence of a Nash equilibrium when the objective function of each follower is convex in its optimizing variable, and the feasible set is compact, convex, and nonempty. 
  We then consider the KKT reformulation of the single leader multi-follower game (henceforth, SLMFG), and show using examples that the solutions of both the problems need not be the same, even when each of the followers' problem is convex. 
  In particular, we show that the global minima of both the problems may differ if the follower's problem does not satisfy the Slater's condition. 
  We then show that the local minima of the SLMFG and its KKT reformulation are the same if, in addition to convexity and Slater's constraints, the local minimum point remains a local minimum for every Lagrange multiplier in each of the followers' problem. 
  Given that this condition is hard to verify in practice, we provide another condition for the local minima of the two problems to be the same using constant rank constraint qualification (CRCQ). 
  We again show using examples that the local optima of the two problems may differ if the conditions are not satisfied.
	
	\medskip
	
	\noindent{\bfseries Keywords:} Bilevel program; Single leader multi-follower game; Constriant qualification; Local solution
	
\end{abstract}

\normalsize

\section{Introduction}

In this paper we consider a bilevel problem with one leader, denoted with~$\leader$, and~$n$ followers, denoted with~$\follower_{1},\dots,\follower_{n}$ respectively. 
Written by~$\Agent := \{\leader,\follower_{1},\dots,\follower_{n}\}$ and~$\Follower := \{\follower_{1},\dots,\follower_{n}\}$, respectively, the set of all players (or agents) and the set of followers. 
Each player~$\agent \in \Agent$ controls his action over the action space~$\Control^{\agent}$ and evaluate his criterion function~$\criterion^{\agent} : \prod_{\bgent \in \Agent}\Control^{\bgent} \to \R$.
We write~$\Control^{\Follower} := \prod_{\follower \in \Follower} \Control^{\follower}$.
For any~$\follower \in \Follower$, we use the notation~$\Control^{-\follower} := \prod_{\gollower \in \Follower \setminus \{\follower\}} \Control^{\gollower}$ whose generic element is denoted by~$\control^{-\follower}$. 
It is typical to represent any vector~$\control^{\Follower} \in \Control^{\Follower}$ with the decomposition~$\control^{\Follower} = (\control^{\follower}, \control^{-\follower}) \in \Control^{\follower} \times \Control^{-\follower}$.
For any given leader's action~$\control^{\leader} \in \Control^{\leader}$, the followers altogether solve a~$\control^{\leader}$-parametrized \emph{Nash equilibrium problem} (briefly, \emph{NEP}): 
Find a \emph{Nash equilibriuam}, i.e., a vector~$\control^{\Follower} \in \Control^{\Follower}$ such that for each~$\follower \in \Follower$, the element~$\control^{\follower}$ solves the problem
\begin{subequations}\label{eq:follower_NEP}
  \begin{empheq}[left=\empheqlbrace]{align}
    \min_{\hat{\control}^{\follower}}\quad & \criterion^{\follower} (\control^{\leader},\hat{\control}^{\follower},\control^{-\follower}) \\
    \text{s.t.}\quad & \hat{\control}^{\follower} \in \Constraint^{\follower}(\control^{\leader}),
  \end{empheq}
\end{subequations}
where 
\begin{alignat*}{2}
  & Y^{\follower}(\control^{\leader}) := \{\control^{\follower} \in \Control^{\follower} \mid \confun^{\follower}_{j}(\control^{\leader},\control^{\follower}) \leq 0 \quad &&(\forall j = 1,\dots,p^{\follower})\}, \\
  & \confun^{\follower}_{j} : \Control^{\leader} \times \Control^{\follower} \to \R  && (\forall j = 1,\dots,p^{\follower}).
\end{alignat*}
denotes the constraint set of the problem of a follower~$\follower \in \Follower$. 
The solution set of this problem with at the leader's decision~$\control^{\leader}~$is denoted with~$\NEP(\control^{\leader})$. 
The task of the leader is to make his \emph{optimal} choice~$\control^{\leader}$ based on his criterion function~$\criterion^{\leader}$ constrained to~$\control^{\Follower} \in \NEP(\control^{\leader})$, i.e.
\begin{subequations}\label{eq:SLMFG}
  \begin{empheq}[left=\empheqlbrace]{align}
    \min_{\control^{\leader},\control^{\Follower}} \quad & \criterion^{\leader}(\control^{\leader},\control^{\Follower}) \\
    \text{s.t.}\quad  & \control^{\leader} \in \Constraint^{\leader} \subset \Control^{\leader} \\
                      & \control^{\Follower} \in \NEP(\control^{\leader}).
  \end{empheq}
\end{subequations} 
This formulation of a \emph{single leader multi-follower game} is commonly known as the standard \emph{optimistic} setting, which corresponds to the case where the  leader minimizes over~$\Constraint^{\leader}$ the optimistic value function~$\psi_{o}(\control^{\leader}) := \inf_{\control^{\Follower} \in \NEP(\control^{\leader})} \criterion^{\leader}(\control^{\leader},\control^{\Follower})$.
On the contrary, one could also consider the \emph{pessimistic} approach to the problem by framing the leader to minimize over~$\Constraint^{\leader}$ the pessimistic value function~$\psi_{p}(\control^{\leader}) := \sup_{\control^{\Follower} \in \NEP(\control^{\leader})} \criterion^{\leader}(\control^{\leader},\control^{\Follower})$.
In this paper, we limited our scope to the standard optimistic formulation as described by~\eqref{eq:SLMFG}.

When the set of follower $\Follower$ is singleton, then the follower's Nash equilibrium problem~\eqref{eq:follower_NEP} becomes an optimization problem.
This special case is referred to as a \emph{bilevel program}.

To ensure that the SLMFG \eqref{eq:SLMFG} is feasible, one may refer, \emph{e.g.}, to the following theorem from \citet{zbMATH03074557} which shows the existence of Nash equilibrium for the followers' problems.
\begin{theorem}[\citet{zbMATH03074557}]
  A Nash equilibrium for \eqref{eq:follower_NEP} exists for every~$\control^{\leader} \in \Control^{\leader}$, if the following assumptions are satisfied:
  \begin{itemize}[label=$\circ$, leftmargin=*]
    \item~$\Constraint^{\follower}(\control^{\leader})$ is compact, convex, and nonempty for all~$\control^{\leader} \in \Control^{\leader}$ and for all~$\follower \in \Follower$,
    \item~$\criterion^{\follower}(\control^{\leader}, \control^{\Follower})$ is continuous in~$\control^{\Follower}$ for every~$\control^{\leader} \in \Constraint^{\leader}$ and for all~$\follower \in \Follower$,
    \item~$\criterion^{\follower}(\control^{\leader},\control^{\follower},\control^{-\follower})$ is convex in~$\control^{\follower}$ for all~$(\control^{\leader},\control^{-\follower})$ and for all~$\follower \in \Follower$.
  \end{itemize}
\end{theorem}

To practically solve a bilevel program or a SLMFG, one requires to reformulate the problem into a single-level problem so that it becomes more numerically tractable.
A classical solution method for~\eqref{eq:SLMFG} involves a reformulation of~$\NEP(\control^{\leader})$ using the concatenated KKT system.
Consider the case that~$\Control_{\agent} = \R^{n_{\agent}}$ for all~$\agent \in \Agent$. 
Suppose that the maps~$\control^{\follower} \mapsto \criterion^{\follower}(\control^{\follower},\control^{-\follower})$ and~$\control^{\follower} \mapsto \confun^{\follower}_{j}(\control^{\leader},\control^{\follower})$ are convex and continuously differentiable for all~$\follower \in \Follower$ and~$j = 1,\dots,p^{\follower}$.
We also assume that all the gradients~$\grad_{\control^{\follower}} \criterion^{\follower}(\control^{\follower},\control^{-\follower})$ and~$\grad_{\control^{\follower}} \confun^{\follower}_{j}(\control^{\leader},\control^{\follower})$ are continuous in all of their variables. 
Then the \emph{Mathematical Program with Complementarity Constraints reformulation} (briefly, \emph{MPCC reformulation}), or the \emph{KKT reformulation}, for~\eqref{eq:SLMFG} is defined by
\begin{subequations}\label{eq:MPCC_reformulation}
  \begin{empheq}[left=\empheqlbrace]{alignat=3}
    \min_{\control^{\leader},\control^{\Follower}}\quad  & \criterion^{\leader}(\control^{\leader},\control^{\Follower}) \\
    \text{s.t.}\quad & \control^{\leader} \in \Constraint^{\leader}  \\
                     & \displaystyle \grad_{\control^{\follower}} \criterion^{\follower}(\control^{\leader},\control^{\follower},\control^{-\follower}) + \sum_{j=1}^{p^{\follower}} \mult^{\follower}_{j} \grad_{\control^{\follower}} \confun^{\follower}_{j} (\control^{\leader},\control^{\follower}) = 0   && ( \forall \follower \in \Follower ) \label{eq:multiplier-1} \\
                     & \confun^{\follower} (\control^{\leader},\control^{\follower}) \leq 0, \quad \mult^{\follower} \geq 0, \quad \langle \mult^{\follower}, \confun^{\follower}(\control^{\leader},\control^{\follower}) \rangle = 0 \quad  && ( \forall \follower \in \Follower ), \label{eq:multiplier-2}
  \end{empheq}
\end{subequations}
where~$\mult^{\follower} = (\mult^{\follower}_{1},\dots,\mult^{\follower}_{p^{\follower}}) \in \R^{p^{\follower}}$ and~$\confun^{\follower}(\control^{\leader},\control^{\follower}) = (\confun^{\follower}_{1}(\control^{\leader},\control^{\follower}),\dots,\confun^{\follower}_{p^{\follower}}(\control^{\leader},\control^{\follower}))$.
We shall adopt the notation~$\mult^{\Follower} = (\mult^{\follower}_{j})_{\substack{\follower \in \Follower \\ j = 1,\dots,p^{\follower}}} \in \R^{p^{\Follower}} = \prod_{\follower \in \Follower} \R^{p^{\follower}}$.
For each~$\follower \in \Follower$, we denote with~$\Mult^{\follower}(\control^{\leader},\control^{\Follower})$ the set of the corresponding multipliers, \emph{i.e.}, of all~$\mult^{\follower} = (\mult^{\follower}_{1},\dots,\mult^{\follower}_{p^{\follower}})$ satisfying~\eqref{eq:multiplier-1} and~\eqref{eq:multiplier-2}.
We also write~$\Mult^{\Follower}(\control^{\leader},\control^{\Follower}) := \prod_{\follower \in \Follower} \Mult^{\follower}(\control^{\leader},\control^{\Follower})$.

In case of a single-follower (\emph{i.e.} the bilevel programs), the relationship betwen the local solution sets of \eqref{eq:SLMFG} and its MPCC reformulation \eqref{eq:MPCC_reformulation} were studied by \citet{zbMATH06008674}.
In particular, they \cite{zbMATH06008674} pointed out that the local solutions of the two problems could be independent and also gave sufficient conditions for which they coincide.
Similar results were studied by \citet{zbMATH07063784} for the pessimistic bilevel programs.

The aim of this paper is to present, in the same spirits as~\citet{zbMATH06008674} and~\citet{zbMATH07063784}, to the optimistic multi-follower setting.
In Section~\ref{sec:main}, we present sufficient conditions under which the local solution set of the SLMFG~\eqref{eq:SLMFG} is included in that of the MPCC problem~\eqref{eq:MPCC_reformulation}, and vice versa.
We also give examples to justify that the result might not be true if some of the assumptions were dropped from our main theorems.
In Section~\ref{sec:GNEP-lower}, we argue that our result is also applicable to the case where the followers might have shared constraints with the structure of~\citet{zbMATH03229265}.
In particular, we show that a Rosen's generealized Nash equilibrium problem of the followers could be equivalently presented as a classical Nash equilibrium problem.
As a consequence, all the results in Section~\ref{sec:GNEP-lower} are still applicable and their assumptions can directly be verified explicitly with the functional description.

\section{KKT Reformulation in Standard Optimistic SLMFG}\label{sec:main}

This section is dedicated to the conditions under which the solutions of the original SLMFG~\eqref{eq:SLMFG} and the corresponding MPCC reformulation~\eqref{eq:MPCC_reformulation} are included in one another.
Each sufficient condition is supplied with an example to emphasize that the assumed properties are required and could not be dropped without losing the desired conclusion.

\begin{theorem}\label{thm:SLMFG->MPCC-local}
  Let~$(\bar{\control}^{\leader},\bar{\control}^{\Follower})$ be a local solution of the SLMFG~\eqref{eq:SLMFG} and let the follower's problem~\eqref{eq:follower_NEP} for every follower~$\follower \in \Follower$ be such that the maps~$\control^{\follower} \mapsto \criterion^{\follower}(\control^{\leader},\control^{\follower},\control^{-\follower})$ and~$\control^{\follower} \mapsto \confun^{\follower}_{j}(\control^{\leader},\control^{\follower})$ are convex with Slater's CQ satisfied at~$\control^{\leader} = \bar{\control}^{\leader}$.
  Then there exist~$\bar{\mult}^{\Follower} = (\mult^{\follower}_{j})_{\substack{\follower \in \Follower \\ j =1,\dots,p^{\follower}}} \geq 0$ such that~$(\bar{\control}^{\leader},\bar{\control}^{\Follower},\bar{\mult}^{\Follower})$ is a local solution of the MPCC problem~\eqref{eq:MPCC_reformulation}.
\end{theorem}
\begin{proof}
  Observe that~$\bar{\control}^{\Follower} \in \NEP(\bar{\control}^{\leader})$ implies
  \[
    \bar{\control}^{\follower} \in \argmin_{\control^{\follower} \in \Constraint^{\follower}(\control^{\leader})} \criterion^{\follower}(\control^{\leader},\control^{\follower},\control^{-\follower}).
  \]
  This, in combination with the convexity of~$\criterion^{\follower}$ and~$\confun^{\follower}$ in~$\control^{\follower}$ and the Slater's CQ at~$\bar{\control}^{\leader}$, yields inf and only if the KKT conditions hold for the follower~$\follower$ problem at~$(\bar{\control}^{\leader},\bar{\control}^{\Follower})$.
  Hence~$\control^{\Follower} \in \NEP(\bar{\control}^{\leader})$ if and only if there exists~$\bar{\mult}^{\Follower} \in \R^{p^{\Follower}}$ such that
  \begin{alignat*}{2}
    &\grad_{\control^{\follower}} \criterion^{\follower}(\bar{\control}^{\leader},\bar{\control}^{\follower},\bar{\control}^{-\follower}) + \sum_{j=1}^{p^{\follower}} \bar{\mult}^{\follower}_{j} \grad_{\control^{\follower}} \confun^{\follower}_{j} (\bar{\control}^{\leader},\bar{\control}^{\follower}) = 0 \quad && ( \forall \follower \in \Follower )  \\
    &\confun^{\follower} (\bar{\control}^{\leader},\bar{\control}^{\follower}) \leq 0, \quad \bar{\mult}^{\follower} \geq 0, \quad \langle \bar{\mult}^{\follower}, \confun^{\follower}(\bar{\control}^{\leader},\bar{\control}^{\follower}) \rangle = 0  \quad  &&( \forall \follower \in \Follower ).
  \end{alignat*}
  This means that~$\bar{\control}^{\Follower} \in \NEP(\bar{\control}^{\leader})$ if and only if~$(\bar{\control}^{\leader},\bar{\control}^{\Follower},\bar{\mult}^{\Follower})$ is a feasible point of~\eqref{eq:MPCC_reformulation}.
  Since the objective functions in~\eqref{eq:SLMFG} and~\eqref{eq:MPCC_reformulation} are the same,~$(\bar{\control}^{\leader},\bar{\control}^{\Follower},\bar{\mult}^{\Follower})$ is a local solution of~\eqref{eq:MPCC_reformulation}.
\end{proof}

Since a global solution is also a local solution, the above theorem is also valid for global solutions of a SLMFG~\eqref{eq:SLMFG} and its corresponding MPCC~\eqref{eq:MPCC_reformulation}.
The following example shows that the Slater's CQ assumption could not be dropped in order to attain the same conclusion.
\begin{example}
  Consider a two-follower setting, \emph{i.e.}~$\Follower = \{\follower_{1},\follower_{2}\}$.
  The leader $\leader$'s decision set is~$\Control^{\leader} = \R$ and both followers make decision over the set~$\Control^{\follower_{1}} = \Control^{\follower_{2}} = \R^{2}$.
  For a given~$\control^{\leader} \in \Control^{\leader}$, the follower $\follower_{1}$ solves the problem
  \begin{subequations}\label{equation:ex-nep-f1}
    \begin{empheq}[left=\empheqlbrace]{align}
      \min_{\control^{\follower_{1}} = (\control^{\follower_{1}}_{1}, \control^{\follower_{1}}_{2})} \quad & \control^{\follower_{1}}_{1} + \control^{\follower_{2}}_{2} \\
      \text{s.t.} \quad & (\control^{\follower_{1}}_{1})^{2} - \control^{\follower_{1}}_{2} \leq \control^{\leader} \\
                        & (\control^{\follower_{1}}_{1})^{2} + \control^{\follower_{1}}_{2} \leq 0,
    \end{empheq}
  \end{subequations}
  while the follower $\follower_{2}$ solves
  \begin{subequations}\label{equation:ex-nep-f2}
    \begin{empheq}[left=\empheqlbrace]{align}
      \min_{\control^{\follower_{2}} = (\control^{\follower_{2}}_{1},\control^{\follower_{2}}_{2})} \quad 
    & \control^{\follower_{1}}_{1} + \control^{\follower_{2}}_{2} \\
      \text{s.t.}\quad &(\control^{\follower_{2}}_{1})^{2} - \control^{\follower_{2}}_{2} \leq \control^{\leader} \\
                       & (\control^{\follower_{2}}_{1})^{2} + \control^{\follower_{2}}_{2} \leq 0.
    \end{empheq}
  \end{subequations}
  The optimal responses of~$\follower_{1}$ and~$\follower_{2}$, given the leader's action~$\control^{\leader} \geq 0$ are given by
  \[
    \bar{\control}^{\follower_{1}} = \bar{\control}^{\follower_{2}} = (-\sqrt{\control^{\leader}/2},-\control^{\leader}/2),
  \]
  with the Lagrange multipliers~$\bar{\mult}^{\follower_{1}}_{1} = \bar{\mult}^{\follower_{1}}_{2} = \bar{\mult}^{\follower_{2}}_{1} = \bar{\mult}^{\follower_{2}}_{2} = \frac{1}{4\sqrt{\control^{\leader}/2}}$ when $\control^{\leader} > 0$.
  At~$\control^{\leader} = 0$, the KKT conditions are not satisfied for the optimal point $(0,0)$, given that the problems are not regular when~$\control^{\leader} = 0$.

  Now consider the leader's problem to be 
  \begin{empheq}[left=\empheqlbrace]{align*}
    \min_{\control^{\leader},\control^{\Follower}} \quad & \control^{\leader} \\
    \text{s.t.} \quad & \control^{\Follower} \in \NEP(\control^{\leader}),
  \end{empheq}
  where in this case~$\NEP(\control^{\leader})$ corresponds to \eqref{equation:ex-nep-f1} and \eqref{equation:ex-nep-f2}.
  The (unique) solution to this SLMFG is~$(\hat{\control}^{\leader},\hat{\control}^{\Follower}) = (0,0,0,0,0)$.
  However, as calculated above, the corresponding MPCC is infeasible for~$\control^{\leader} = 0$.
\end{example}

The following theorem concerns with the converse of Theorem~\ref{thm:SLMFG->MPCC-local},
which together would gives the equivalence between global solutions of a SLMFG~\eqref{eq:SLMFG} and its MPCC reformulation~\eqref{eq:MPCC_reformulation}.
\begin{theorem}\label{thm:MPCC->SLMFG-global}
  Let~$(\bar{\control}^{\leader},\bar{\control}^{\Follower}\bar{\mult}^{\Follower})$ be a global solution of the MPCC~\eqref{eq:MPCC_reformulation} and let the follower's problem~\eqref{eq:follower_NEP} for all~$\follower \in \Follower$ be such that the maps~$\control^{\follower} \mapsto \criterion^{\follower}(\control^{\leader},\control^{\follower},\control^{-\follower})$ and~$\control^{\follower} \mapsto \confun^{\follower}_{j}(\control^{\leader},\control^{\follower})$ are convex with Slater's CQ satisfied at every~$\control^{\leader} \in \Constraint^{\leader}$.
  Then~$(\bar{\control}^{\leader},\bar{\control}^{\Follower})$ is a global solution of the SLMFG~\eqref{eq:SLMFG}.
\end{theorem}
\begin{proof}
  Assume that there exists~$\hat{\control}^{\leader} \in \Constraint^{\leader}$ and~$(\hat{\control}^{\Follower}) \in \NEP(\hat{\control}^{\leader})$ such that
  \[
    \criterion^{\leader}(\hat{\control}^{\leader},\hat{\control}^{\Follower}) < \criterion^{\leader}(\bar{\control}^{\leader},\bar{\control}^{\Follower}).
  \]
  Given that~$\criterion^{\follower}$ and~$\confun^{\follower}_{j}$ are all convex and Slater's CQ is satisfied at~$\hat{\control}^{\leader}$, the KKT conditions hold for each follower~$\follower \in \Follower$.
  This implies that~$(\hat{\control}^{\leader},\hat{\control}^{\Follower},\hat{\mult}^{\Follower})$ is feasible for the MPCC~\eqref{eq:MPCC_reformulation} for some multiplier~$\hat{\mult}^{\Follower} \geq 0$, which in turn implies that~$(\bar{\control}^{\leader},\bar{\control}^{\Follower}\bar{\mult}^{\Follower})$ is not a global solution of the MPCC~\eqref{eq:MPCC_reformulation}.
\end{proof}

The equivalence of global solutions of the SLMFG~\eqref{eq:SLMFG} and the MPCC problem~\eqref{eq:MPCC_reformulation} was proved under a different set of conditions in \citet[Theorem 26]{svensson:tel-02892212}.
In \cite{svensson:tel-02892212}, it was shown that if (1) the map~$\control^{\follower} \mapsto \criterion^{\follower}(\control^{\leader},\control^{\follower},\control^{-\follower})$ is convex, (2) the maps~$(\control^{\leader},\control^{\follower}) \mapsto \max_{j=1,\dots,p^{\follower}} \confun^{\follower}_{j} (\control^{\leader},\control^{\follower})$ is jointly convex for all~$\follower \in \Follower$, and (3) for all~$\follower \in \Follower$, there exists a~$(\control^{\leader}(\follower),\control^{\follower})$ such that~$\max_{j=1,\dots,p^{\follower}} \confun^{\follower}_{j}(\control^{\leader}(\follower), \control^{\follower}) < 0$, then~$(\bar{\control}^{\leader},\bar{\control}^{\Follower},\bar{\mult}^{\Follower})$ is a global solution of~\eqref{eq:MPCC_reformulation} implies~$(\bar{\control}^{\leader},\bar{\control}^{\follower})$ is a global solution of~\eqref{eq:SLMFG}.

Observe that the condition of convexity assumption used in \cite{svensson:tel-02892212} is stronger that our Theorem~\ref{thm:MPCC->SLMFG-global} but the condition on the Slater's CQ is weaker.
This means that our proposed result and the one of \citet{svensson:tel-02892212} provide alternative conditions for the equivalence of global solutions of the two problems.
An important question is whether the equivalence holds under weaker conditions.
In particular, we ask whether the equivalence holds if~$\control^{\follower} \mapsto \criterion^{\follower}(\control^{\leader},\control^{\follower})$ is convex, and for each~$\follower \in \Follower$, there exists a~$(\control^{\leader}(\follower),\control^{\follower})$ such that~$\max_{j=1,\dots,p^{\follower}} \confun^{\follower}_{j}(\control^{\leader}(\follower),\control^{\follower}) < 0$.
The following example provides a negative answer to this question.
\begin{example}
  Consider a two-follower setting, \emph{i.e.}~$\Follower = \{\follower_{1},\follower_{2}\}$.
  The leader $\leader$'s decision set is~$\Control^{\leader} = \R$ and both followers make decision over the set~$\Control^{\follower_{1}} = \Control^{\follower_{2}} = \R$.
  For a given~$\control^{\leader} \in \Control^{\leader}$, the follower $\follower_{i}$ ($i = 1,2$) solves the problem
  \begin{subequations}\label{equation:ex2-lower}
    \begin{empheq}[left=\empheqlbrace]{align}
      \min_{\control^{\follower_{i}}} \quad & \control^{\leader}(\control^{\follower_{1}} + \control^{\follower_{2}}) \\
      \text{s.t.} \quad & \control^{\leader}(\control^{\follower_{i}})^{2} \leq 0.
    \end{empheq}
  \end{subequations}
  Observe that the constraint is convex in~$\control^{\follower_{i}}$ but not jointly.
  Furthermore, there exists~$\control^{\leader} = \control^{\follower_{1}} = \control^{\follower_{2}} = -1$ which makes $\control^{\leader}(\control^{\follower_{i}})^{2} = -1 < 0$.

  Suppose that the leader's problem is given by
  \begin{subequations}\label{equation:ex-slmfg}
    \begin{empheq}[left=\empheqlbrace]{align}
      \min_{\control^{\leader}, \control^{\follower_{1}}, \control^{\follower_{2}}} \quad & (\control^{\leader} - 1)^{2} + (\control^{\follower_{1}})^{2} + (\control^{\follower_{2}})^{2} \\
      \text{s.t.}\quad & (\control^{\follower_{1}},\control^{\follower_{2}}) \in \NEP(\control^{\leader}),
    \end{empheq}
  \end{subequations}
  where~$\NEP(\control^{\leader})$ refers to the solution set of~\eqref{equation:ex2-lower}.
  The unique global solution of this SLMFG is $(\bar{\control}^{\leader}, \bar{\control}^{\follower_{1}}, \bar{\control}^{\follower_{2}}) = (1,0,0)$.

  On the other hand, the MPCC formulation of~\eqref{equation:ex-slmfg} reads
  \begin{subequations}\label{equation:ex-mpcc}
    \begin{empheq}[left=\empheqlbrace]{align}
      \min_{\control^{\leader}, \control^{\follower_{1}}, \control^{\follower_{2}}, \mult^{\follower_{1}},\mult^{\follower_{2}}} \quad & (\control^{\leader} - 1)^{2} + (\control^{\follower_{1}})^{2} + (\control^{\follower_{2}})^{2} \\
      \text{s.t.}\quad & \control^{\leader} + 2\mult^{\follower_{i}} \control^{\leader} \control^{\follower_{i}} \qquad (i=1,2) \\
                       & \control^{\leader}(\control^{\follower_{i}})^{2} \leq 0, \quad \mult^{\follower_{i}} \geq 0, \quad \mult^{\follower_{i}}\control^{\leader}(\control^{\follower_{i}})^{2} = 0, \qquad i=1,2.
    \end{empheq}
  \end{subequations}
  Every feasible point of this MPCC is of the form $(0,\control^{\follower_{1}},\control^{\follower_{2}},\mult^{\follower_{1}},\mult^{\follower_{2}})$.
  This implies that the solution~$(\bar{\control}^{\leader}, \bar{\control}^{\follower_{1}}, \bar{\control}^{\follower_{2}})$ is not feasible for~\eqref{equation:ex-mpcc}.
  Conversely, the vector~$(0,0,0,\mult^{\follower_{1}},\mult^{\follower_{2}})$ is an optimal point of the MPCC~\eqref{equation:ex-mpcc} but not optimal for the SLMFG~\eqref{equation:ex-slmfg}.
\end{example}

We presented already two examples where the global solutions of the SLMFG and the MPCC turn out to be different.
We next turn to investigating the relation between their local solutions.
First, we define the set of multipliers satisfying the KKT conditions for each follower~$\follower \in \Follower$ at~$(\control^{\leader},\control^{\Follower})$ as
\[
  \Mult^{\follower}(\control^{\leader},\control^{\Follower}) := \left\{
    \mult^{\follower} \in \R_{+}^{p^{\follower}} \ \left| \
    \begin{array}{c}\displaystyle
      \grad_{\control^{\follower}} \criterion^{\follower}(\control^{\leader},\control^{\follower},\control^{-\follower}) + \sum_{j=1}^{p^{\follower}} \mult^{\follower}_{j} \confun^{\follower}_{j}(\control^{\leader},\control^{\follower}) = 0 \\[.8em]
      \confun^{\follower}_{j} (\control^{\leader},\control^{\follower}) \leq 0, \quad \langle \mult^{\follower}, g^{\follower} (\control^{\leader},\control^{\follower}) = 0 
    \end{array}
\right.\right\}.
\]
According to \citet{Robinson1982}, this set-valued map~$\Mult^{\follower}(\cdot)$ is upper semicontinuous.
Moreover, we defin the set~$\Mult^{\Follower}(\cdot)$ is then defined as the product~$\prod_{\follower \in \Follower} \Mult^{\follower}(\cdot)$.
We have now the following theorem.
\begin{theorem}\label{thm:MPCC->SLMFG-local}
  Let the follower's problem~\eqref{eq:follower_NEP} for every follower~$\follower \in \Follower$ be such that the maps~$\control^{\follower} \mapsto \criterion^{\follower}(\control^{\leader},\control^{\follower},\control^{-\follower})$ and~$\control^{\follower} \mapsto \confun^{\follower}_{j}(\control^{\leader},\control^{\follower})$ are convex with Slater's CQ satisfied at~$\control^{\leader} = \bar{\control}^{\leader}$.
  Suppose that~$( \bar{\control}^{\leader}, \bar{\control}^{\Follower} ) \in \Control^{\leader} \times \Control^{\Follower}$ and for any~$\bar{\mult}^{\Follower} \in \Mult^{\Follower}$, the vector~$( \bar{\control}^{\leader}, \bar{\control}^{\Follower}, \bar{\mult}^{\Follower} )$ is a local solution of the MPCC problem~\eqref{eq:MPCC_reformulation}.
  Then~$( \bar{\control}^{\leader}, \bar{\control}^{\Follower} )$ is a local solution of the SLMFG~\eqref{eq:SLMFG}.
\end{theorem}
\begin{proof}
  Assume that~$(\bar{\control}^{\leader},\bar{\control}^{\Follower})$ is not a local minimum of~\eqref{eq:SLMFG}.
  Then there exists a sequence~$\{(\control^{\leader}_{k}, \control^{\Follower}_{k})\}$ such that
  \begin{subequations}
    \begin{alignat}{2}
      &(\control^{\leader}_{k}, \control^{\Follower}_{k}) \to (\bar{\control}^{\leader}, \bar{\control}^{\Follower}), \\
      &\control^{\leader}_{k} \in \Constraint^{\leader}, \quad (\control^{\Follower}_{k}) \in \NEP(\control^{\leader}_{k}) \quad &&(\forall k = 1,2,\dots) \\
      &\criterion^{\leader}(\control^{\leader}_{k}, \control^{\Follower}_{k}) < \criterion^{\leader} (\bar{\control}^{\leader}, \bar{\control}^{\Follower}) && (\forall k = 1,2,\dots). \label{eq:lessthan}
    \end{alignat}
  \end{subequations}
  Let us consider the problem of the follower~$\follower$.
  Given that the Slater's CQ is satisfied at~$\control^{\leader} = \bar{\control}^{\leader}$, we assert that it also satisfies at a neighborhood around~$\bar{\control}^{\leader}$.
  Without loss of generality, we may assume that the sequence~$\{\control^{\leader}_{k}\}$ belongs to this neighborhood.
  Now for each~$(\control^{\leader}_{k},\control^{\Follower}_{k})$, there exists a multiplier~$\mult^{\follower}_{k} \in \Mult^{\follower}(\control^{\leader}_{k},\control^{\Follower}_{k})$ having an accumulation point~$\bar{\mult}^{\follower} \in \Mult^{\follower}(\bar{\control}^{\leader},\bar{\control}^{\Follower})$ by the upper semicontinuity of~$\Mult^{\follower}(\cdot)$ \cite{Robinson1982}.
  Extracting a subsequence if neccessary, we may assume that the sequence~$\{(\control^{\leader}_{k},\control^{\Follower}_{k},\mult^{F}_{k})\}$ of feasible points of~\eqref{eq:MPCC_reformulation} is convergent to the point~$(\bar{\control}^{\leader},\bar{\control}^{\Follower},\bar{\mult}^{\Follower})$ which is also feasible to~\eqref{eq:MPCC_reformulation}.
  In view of~\eqref{eq:lessthan}, the point~$(\bar{\control}^{\leader},\bar{\control}^{\Follower})$ is not a local solution of the MPCC~\eqref{eq:MPCC_reformulation} and the theorem is thus proved.
\end{proof}

We insist here that the adjoint vector~$(\bar{\control}^{\leader},\bar{\control}^{\Follower},\bar{\mult}^{\Follower})$ is a local minimum of a MPCC problem for the whole set~$\Mult^{\Follower}(\bar{\control}^{\leader},\bar{\control}^{\Follower})$ in the above Theorem~\ref{thm:MPCC->SLMFG-local}.
In fact, when such a condition fails, the sets of local solutions of a MPCC and a SLMFG may differ even when their gloabal solution sets are identical.
This is shown in the following example.
\begin{example}
  Consider a two-follower setting, with~$\Follower = \{\follower_{1},\follower_{2}\}$ and~$\Control^{\leader} = \Control^{\follower_{1}} = \Control^{\follower_{2}} = \R$.
  Suppose that for any~$\control^{\leader} \in \Control^{\leader}$ and~$i = 1,2$, the follower~$\follower_{i}$'s problem is
  \begin{subequations}\label{eq:follower_exmp_same_global_diff_local}
    \begin{empheq}[left=\empheqlbrace]{alignat=3}
      \min_{\control^{\follower_{i}}} \quad & -&&\control^{\follower_{1}} - \control^{\follower_{2}} \\
      \text{s.t.} \quad & &&\control^{\leader} + \control^{\follower_{i}} \leq 1 \\
                        & -&&\control^{\leader} + \control^{\follower_{i}} \leq 1.
    \end{empheq}
  \end{subequations}
  The solutions~$(\bar{\control}^{\follower_{1}}, \bar{\control}^{\follower_{2}})$ of this Nash equilibrium problem, together with the corresponding multipliers~$(\bar{\mult}^{\follower_{1}},\bar{\mult}^{\follower_{2}})$, are given by
  \[
    \bar{\control}^{\follower_{i}} = \begin{cases}
      1 + \control^{\leader}  & \text{if~$\control^{\leader} \leq 0$,} \\
      1 - \control^{\leader}  & \text{if~$\control^{\leader} \geq 0$,} \\
    \end{cases} \qquad
    \bar{\mult}^{\follower_{i}} = \begin{cases}
      (0,1) & \text{if~$x < 0$,} \\
      (1,0) & \text{if~$x > 0$,} \\
      \conv\{(0,1),(1,0)\} & \text{if~$x = 0$.}
    \end{cases}
  \]
  Now, suppose that the leader's problem is
  \begin{subequations}\label{eq:leader_exmp_same_global_diff_local}
    \begin{empheq}[left=\empheqlbrace]{align}
      \min_{\control^{\leader},\control^{\follower_{1}},\control^{\follower_{2}}} \quad & (\control^{\leader} - 1)^{2} + (\control^{\follower_{1}} - 1)^{2} + (\control^{\follower_{2}} - 1)^{2} \\
      \text{s.t.} \quad & (\control^{\follower_{1}},\control^{\follower_{2}}) \in \NEP(\control^{\leader}),
    \end{empheq}
  \end{subequations}
  where~$\NEP(\control^{\leader})$ is the set of Nash equilibria of~\eqref{eq:follower_exmp_same_global_diff_local}.
  The SLMFG~\eqref{eq:leader_exmp_same_global_diff_local} has a unique global optimal point~$(\hat{\control}^{\leader},\hat{\control}^{\follower_{1}},\hat{\control}^{\follower_{2}}) = (\frac{1}{3},\frac{2}{3},\frac{2}{3})$ and no other local solutions.
  However, the corresponding MPCC has another local optimal solution, which is
  \[
    (\tilde{\control}^{\leader}, \tilde{\control}^{\follower_{1}}, \tilde{\control}^{\follower_{2}}, \tilde{\mult}^{\follower_{1}}, \tilde{\mult}^{\follower_{2}})
    = (0,1,1,(0,1),(0,1)).
  \]
  Nevertheless if we choose another multiplier, say~$\tilde{\tilde{\mult}}^{\follower_{1}} = \tilde{\tilde{\mult}}^{\follower_{2}} = (1,0)$, then the vector
  \[
    (\tilde{\control}^{\leader}, \tilde{\control}^{\follower_{1}}, \tilde{\control}^{\follower_{2}}, \tilde{\tilde{\mult}}^{\follower_{1}}, \tilde{\tilde{\mult}}^{\follower_{2}})
    = (0,1,1,(1,0),(1,0))
  \]
  is no longer a local solution of the same MPCC.
  This example thus illustrate that if a point fails to be a local minimum for every multipliers, the equivalence of local minima between a SLMFG and its MPCC reformulation could not be guaranteed.
\end{example}

The condition in Theorem~\ref{thm:MPCC->SLMFG-local} is hard to verify in practice since it involves verifying whether the local solution of~\eqref{eq:MPCC_reformulation} remains the local solution for uncountable points. 
In the following theorem, we propose a simplified process that involves only multipliers on the vertices but with the discrepancy of an additional {\em constant rank constraint qualification} (CRCQ).
We first define CRCQ.
\begin{definition}
  The follower's problem for the follower~$\follower \in \Follower$ (as in~\eqref{eq:follower_NEP}) satisfies the \emph{constant rank constraint qualification} (briefly, \emph{CRCQ}), at~$(\bar{\control}^{\leader},\bar{\control}^{\follower})$ if there exists a neighborhood~$V^{\follower}$ of~$(\bar{\control}^{\leader},\bar{\control}^{\follower})$ such that for each active index subset 
  \[
    I^{\follower} \subset \Active^{\follower}(\bar{\control}^{\leader},\bar{\control}^{\follower}) := \{j \mid \confun^{\follower}_{j}(\bar{\control}^{\leader},\bar{\control}^{\follower}) = 0\},
  \]
  the rank of the matrix~$(\grad_{\control^{\follower}} \confun^{\follower}_{j}(\control^{\leader},\control^{\follower}))_{j \in I^{\follower}}$ is a constant for~$(\control^{\leader},\control^{\follower}) \in V_{i}$.
\end{definition}

Now we are ready to show the simplified verification process to the vertex multipliers.
Note that a part of the proof borrows an idea from the simplex method to characterize the vertices of a set defined by linear constraints using basic and nonbasic variables.
We refer the reader to~\cite{zbMATH01032354} for more details.
\begin{theorem}
  Let the follower's problem~\eqref{eq:follower_NEP} for every follower~$\follower \in \Follower$ be such that the maps~$\control^{\follower} \mapsto \criterion^{\follower}(\control^{\leader},\control^{\follower},\control^{-\follower})$ and~$\control^{\follower} \mapsto \confun^{\follower}(\control^{\leader},\control^{\follower})$ are convex with Slater's CQ at~$\control^{\leader} = \bar{\control}^{\leader}$.
  Let~$(\bar{\control}^{\leader}, \bar{\control}^{\Follower},\bar{\mult}^{\Follower})$ be a local minimum of the MPCC problem~\eqref{eq:MPCC_reformulation} for all vertex points~$\bar{\mult}^{\Follower}$ of~$\Mult^{\Follower}(\bar{\control}^{\leader},\bar{\control}^{\Follower})$ and that the follower's problem satisfies the CRCQ at~$(\bar{\control}^{\leader},\bar{\control}^{\follower})$ for all~$\follower \in \Follower$.
  Then~$(\bar{\control}^{\leader},\bar{\control}^{\Follower})$ is a local minimum of the SLMFG~\eqref{eq:SLMFG}.
\end{theorem}
\begin{proof}
  Assume to the contrary that~$(\bar{\control}^{\leader},\bar{\control}^{\Follower})$ is not a local minimum of the problem SLMFG~\eqref{eq:SLMFG}.
  Then there exists a sequence~$\{(\control^{\leader}_{k},\control^{\Follower}_{k}j)\}$ such that
  \begin{subequations}\label{eq:sequence_construction_CRCQ}
    \begin{alignat}{2}
      &(\control^{\leader}_{k}, \control^{\Follower}_{k}) \to (\bar{\control}^{\leader}, \bar{\control}^{\Follower}), \label{eq:convergent} \\
      &\control^{\leader}_{k} \in \Constraint^{\leader}, \quad (\control^{\Follower}_{k}) \in \NEP(\control^{\leader}_{k}), \quad  \quad &&(\forall k = 1,2,\dots) \\
      &\Active^{\follower}(\control^{\leader}_{k},\control^{\follower}_{k}) \subset \Active^{\follower}(\bar{\control}^{\leader}, \bar{\control}^{\follower}) \quad && (\forall k = 1,2,\dots) \\
      &\criterion^{\leader}(\control^{\leader}_{k}, \control^{\Follower}_{k}) < \criterion^{\leader} (\bar{\control}^{\leader}, \bar{\control}^{\Follower}) && (\forall k = 1,2,\dots). \label{eq:lessthan2}
      \intertext{and that both of the following are constants over all~$k = 1,2,\dots$:}
      &\Active^{\follower}(\control^{\leader}_{k},\control^{\follower}_{k}) =: \mathcal{M}, \\
      &\mathrm{rank}(\grad_{\control^{\follower}} \confun^{\follower}_{j}(\control^{\leader}_{k},\control^{\follower}_{k}))_{j \in \Active^{\follower}(\control^{\leader}_{k},\control^{\follower}_{k})} =: \rho^{\follower}.
    \end{alignat}
  \end{subequations}
  Putting~$M := \abs{\mathcal{M}}$, we shall enumerate the active set as~$\Active^{\follower}(\control^{\leader}_{k},\control^{\follower}_{k}) = \{j_{1},\dots,j_{M}\}$.

  Consider any follower~$\follower$'s problem.
  Knowing that the Slater's CQ is satisfied at~$\control^{\leader} = \bar{\control}^{\leader}$, we can assert that it is also satisfied at a neighborhood about~$\bar{\control}^{\leader}$.
  We assume, without loss of generality, that~$\{\control^{\leader}_{k}\}$ belongs to this neighborhood.
  Now, for each~$(\control^{\leader}_{k},\control^{\Follower}_{k})$ in the sequence in~\eqref{eq:sequence_construction_CRCQ}, there exists a vertex~$\mult^{\follower}_{k} \in \Mult^{\follower}(\control^{\leader}_{k}, \control^{\Follower}_{k})$ that forms a sequence~$\{\mult^{\follower}_{k}\}$ having an accumulation point~$\bar{\mult}^{\follower} \in \Mult^{\follower}(\bar{\control}^{\leader}, \bar{\control}^{\Follower})$ by the upper semicontinuity of the Lagrange multiplier map~\cite{Robinson1982}.
  It remains to show that~$\bar{\mult}^{\follower}$ is a vertex of~$\Mult^{\follower}(\bar{\control}^{\leader},\bar{\control}^{\Follower})$.

  Given that the rank of the matrix~$\grad_{\control^{\follower}} \confun^{\follower}_{j}(\control^{\leader}_{k},\control^{\follower}_{k}))_{j \in \Active^{\follower}(\control^{\leader}_{k},\control^{\follower}_{k})}$ equals~$\rho^{\follower}$, there are~$(M - \rho^{\follower})$ redundant constraints among the total of~$M$ constraints.
  On removing those constraints, we have a resulting matrix of size~$\rho^{\follower} \times M$.
  Let us denote such a matrix with~$A^{\follower}$.
  Also denote with~$\mult^{\follower}_{k}|_{\Active^{\follower}(\control^{\leader}_{k})} := (\mult^{\follower}_{k,j_{1}},\cdots,\mult^{\follower}_{k,j_{M}})$ the components in~$\Active^{\follower}(\control^{\leader}_{k},\control^{\Follower}_{k})$ of the vector~$\mult^{\follower}_{k}$, and write~$b^{\follower} := -\grad_{\control^{\follower}} \criterion^{\follower}(\control^{\leader}_{k},\control^{\follower}_{k},\control^{-\follower}_{k})$.
  Then the following conditions hold:
  \begin{subequations}
    \begin{alignat}{2}
      A^{\follower} \mult^{\follower}_{k}|_{\Active^{\follower}} &= b^{\follower}, \quad \label{eq:linsys_CRCQ}\\
      \mult^{\follower}_{k,j} &\geq 0, \quad && (\forall j \in \Active^{\follower}(\control^{\leader}_{k},\control^{\follower}_{k})), \\
      \mult^{\follower}_{k,j} &= 0, \quad &&(\forall j \not\in \Active^{\follower}(\control^{\leader}_{k},\control^{\follower}_{k})).
    \end{alignat}
  \end{subequations}
  In fact, the left-hand-side of the linear system~\eqref{eq:linsys_CRCQ} could be decomposed into
  \[
    B^{\follower}\mult^{\follower}_{k}|_{\Active^{\follower}}^{B} + N^{\follower}\mult^{\follower}_{k}|_{\Active^{\follower}}^{N},
  \]
  where~$\mult^{\follower}_{k}|_{\Active^{\follower}}^{B}$ and~$\mult^{\follower}_{k}|_{\Active^{\follower}}^{N}$ correspond, respectively, to the basic and nonbasic parts of~$\mult^{\follower}_{k}|_{\Active^{\follower}}$.
  Similarly,~$B^{\follower}$ and~$N^{\follower}$ above refer to the corresponding columns to such basic and nonbasic parts.
  Now given that the number of nonbasic variables equals the constant~$M-\rho^{\follower}_{k}$, we may restrict to a subsequence of~$\{\mult^{\follower}_{k}\}$, also denote with~$\{\mult^{\follower}_{k}\}$ for convenience, such that the corresponding sets of basic part are identical for all~$k \in \N$.
  The set of basic and nonbasic variables thus remain the same also for the accumulation point~$\bar{\mult}^{\follower}$, which implies that~$\bar{\mult}^{\follower}$ is a vertex point of~$\Active^{\follower}(\bar{\control}^{\leader}, \bar{\control}^{\Follower})$. Since the choice of~$\follower \in \Follower$ was arbitrary, it justifies that~$(\control^{\leader}_{k},\control^{\Follower}_{k},\mult^{\Follower}_{k})$ is a feasible point of the MPCC~\eqref{eq:MPCC_reformulation} with~$\mult^{\follower}_{k}$ being a vertex point of~$\Mult^{\follower}(\control^{\leader}_{k},\control^{\follower}_{k})$ for all~$k \in \N$ and all~$\follower \in \Follower$.
  Together with~\eqref{eq:convergent} and~\eqref{eq:lessthan2}, this is a contradiction with the hypothesis that~$(\bar{\control}^{\leader}, \bar{\control}^{\Follower}, \bar{\mult}^{\Follower})$ is a local solution of the MPCC~\eqref{eq:MPCC_reformulation}.
  We therefore conclude that~$(\bar{\control}^{\leader}, \bar{\control}^{\Follower})$ is a local solution of SLMFG~\eqref{eq:SLMFG}.
\end{proof}

The following example shows tha the CRCQ assumption is crucial.
In particular, it illustrates that the local solutions of the SLMFG~\eqref{eq:SLMFG} and the MPCC reformulation~\eqref{eq:MPCC_reformulation} need not be the same if the follower's problem does not satisfy the CRCQ at $(\bar{\control}^{\leader}, \bar{\control}^{\Follower})$.
\begin{example}
  Consider a SLMFG with two followers, \emph{i.e.}~$\Follower = \{\follower_{1},\follower_{2}\}$, with~$\Control^{\leader} = \R$ and~$\Control^{\follower_{1}} = \Control^{\follower_{2}} = \R^{2}$.
  Suppose that a follower~$\follower_{i}$ solves, given~$\control^{\leader} \in \Control^{\leader}$, the following problem
  \begin{subequations}\label{eq:follower_exmp_no_crcq}
    \begin{empheq}[left=\empheqlbrace]{align}
      \min_{\control^{\follower_{i}}_{1},\control^{\follower_{i}}_{2}} \quad 
        & \sum_{j=1}^{2} [ (\control^{\follower_{j}}_{1})^{2} + (\control^{\follower_{j}}_{2} + 1)^{2}] \\
        \text{s.t.} \quad
        & (\control^{\follower_{i}}_{1} - \control^{\leader})^{2} + (\control^{\follower_{i}}_{2} - \control^{\leader} - 1)^{2} - 1 \leq 0 \\
        & (\control^{\follower_{i}}_{1} + \control^{\leader})^{2} + (\control^{\follower_{i}}_{2} - \control^{\leader} - 1)^{2} - 1 \leq 0.
    \end{empheq}
  \end{subequations}
  Now, let the leader~$\leader$'s problem be
  \begin{subequations}\label{eq:leader_exmp_no_crcq}
    \begin{empheq}[left=\empheqlbrace]{align}
      \min_{\control^{\leader},\control^{\Follower}} \quad
        & - \control^{\follower_{1}}_{2} - \control^{\follower_{2}}_{2} \\
        \text{s.t.} \quad
        & 0 \leq \control^{\leader} \leq \frac{1}{2} \\
        & \control^{\Follower} = (\control^{\follower_{1}},\control^{\follower}_{2}) \in \NEP(\control^{\leader}),
    \end{empheq}
  \end{subequations}
  where~$\NEP(\control^{\leader})$ denotes the set of Nash equilibria of~\eqref{eq:follower_exmp_no_crcq}.
  Indeed, replacing the~$\NEP(\control^{\leader})$ with its concatenated KKT conditions, one ends up with the MPCC reformulation of~\eqref{eq:leader_exmp_no_crcq} as follows
  \begin{subequations}\label{eq:exmp_crcq_MPCC}\allowdisplaybreaks
    \begin{empheq}[left=\empheqlbrace]{alignat=2}
      \min_{\control^{\leader},\control^{\Follower}} \quad
        & - \control^{\follower_{1}}_{2} - \control^{\follower_{2}}_{2} \\
        \text{s.t.} \quad
        & 0 \leq \control^{\leader} \leq \frac{1}{2} \\
        & \forall i = 1,2: \\
        &  \quad 2\control^{\follower_{i}}_{1} + \mult^{\follower_{i}}_{1}[2(\control^{\follower_{i}}_{1} - \control^{\leader})] + \mult^{\follower_{i}}_{2}[2(\control^{\follower_{i}}_{1} + \control^{\leader})] = 0   \label{eq:exmp_mpcc_crcq_key2}  \\
        &  \quad 2(\control^{\follower_{i}}_{2} + 1) + \mult^{\follower_{i}}_{1}[2(\control^{\follower_{i}}_{2} - \control^{\leader} - 1)] + \mult^{\follower_{i}}_{2}[2(\control^{\follower_{i}}_{2} - \control^{\leader} - 1)] = 0   \label{eq:exmp_mpcc_crcq_key1}\\
        &  \quad \mult^{\follower_{i}}_{1}[(\control^{\follower_{i}}_{1} - \control^{\leader})^{2} + (\control^{\follower_{i}}_{2} - \control^{\leader} - 1)^{2} - 1] = 0  \\
        &  \quad \mult^{\follower_{i}}_{2}[(\control^{\follower_{i}}_{1} + \control^{\leader})^{2} + (\control^{\follower_{i}}_{2} - \control^{\leader} - 1)^{2} - 1] = 0  \\
        &  \quad \mult^{\follower_{i}}_{1} \geq 0, \quad (\control^{\follower_{i}}_{1} - \control^{\leader})^{2} + (\control^{\follower_{i}}_{2} - \control^{\leader} - 1)^{2} - 1 \leq 0  \\
        &  \quad \mult^{\follower_{i}}_{2} \geq 0, \quad (\control^{\follower_{i}}_{1} + \control^{\leader})^{2} + (\control^{\follower_{i}}_{2} - \control^{\leader} - 1)^{2} - 1 \leq 0.
    \end{empheq}
  \end{subequations}
  We take~$\bar{\control}^{\leader} = 0$ and~$\bar{\control}^{\follower_{1}} = \bar{\control}^{\follower_{2}} = (0,0)$.
  Then~$(\bar{\control}^{\leader},\bar{\control}^{\follower_{1}},\bar{\control}^{\follower_{2}}) = (0,(0,0),(0,0))$ is feasible for the SLMFG~\eqref{eq:leader_exmp_no_crcq}, but not optimal.
  Looking at~\eqref{eq:exmp_mpcc_crcq_key1}, one obtain the set of multipliers as
  \begin{equation}\label{eq:exmp_mpcc_crcq_Mult}
    \Mult^{\follower_{i}} (\bar{\control}^{\leader},\bar{\control}^{\follower_{1}},\bar{\control}^{\follower_{2}}) = \{ (\mult^{\follower_{i}}_{1},\mult^{\follower_{i}}_{2}) \in \R_{+}^{2} \mid \mult^{\follower_{i}}_{1} + \mult^{\follower_{i}}_{2} = 1 \}.
  \end{equation}

  On the other hand, let~$(\hat{\control}^{\leader},\hat{\control}^{\follower_{1}},\hat{\control}^{\follower_{2}})$ be any feasible point of~\eqref{eq:leader_exmp_no_crcq} with~$\hat{\control}^{\leader} > 0$.
  Then one has~$\hat{\control}^{\follower_{1}}_{1} = \hat{\control}^{\follower_{2}}_{1} = 0$, and~$\hat{\control}^{\follower_{1}},\hat{\control}^{\follower_{2}} \neq 0$.
  Hence, the objective value at such a feasible point is strictly negative.
  We derive, using~\eqref{eq:exmp_mpcc_crcq_key2}, that the corresponding multipliers~$(\hat{\mult}^{\follower_{i}}_{1},\hat{\mult}^{\follower_{i}}_{2}) \in \Mult^{\follower_{i}}(\hat{\control}^{\leader},\hat{\control}^{\follower_{1}},\hat{\control}^{\follower_{2}})$ satisfy~$\mult^{\follower_{i}}_{1} = \mult^{\follower_{i}}_{2} =: \hat{\mult}_{i}$ for~$i=1,2$.
  The condition~\eqref{eq:exmp_mpcc_crcq_key1} implies that
  \begin{equation}\label{eq:exmp_mpcc_crcq_mult_form}
    \hat{\mult}_{i} = -\frac{1}{2} \cdot \frac{\hat{\control}^{\follower_{i}}_{2} + 1}{\hat{\control}^{\follower_{i}}_{2} - \hat{\control}^{\leader} - 1} > 0.
  \end{equation}
  Now, take a sequence~$(\control^{\leader,k},\control^{\follower_{1},k},\control^{\follower_{2},k},\mult^{\follower_{1},k},\mult^{\follower_{2},k})_{k \in \N}$ such that
  \begin{alignat*}{2}
    & \text{$\control^{\leader,k} > 0$ and $(\control^{\leader,k}, \control^{\follower_{1},k}, \control^{\follower_{2},k}) \in \Control^{\leader} \times \Control^{\follower_{1}} \times \Control^{\follower_{2}}$ is feasible,} \quad && \forall k \in \N, \\
    & \mult^{\follower_{i},k} = (\mult^{\follower_{i},k}_{1}, \mult^{\follower_{i},k}_{2}) \in \Mult^{\follower_{i}}(\control^{\leader,k}, \control^{\follower_{1},k}, \control^{\follower_{2},k}), \quad && \forall k \in \N, \\
    & (\control^{\leader,k}, \control^{\follower_{1},k}, \control^{\follower_{2},k}) \to (\bar{\control}^{\leader}, \bar{\control}^{\follower_{1}}, \bar{\control}^{\follower_{2}}) = (0, (0,0), (0,0)).
  \end{alignat*}
  Knowing that~$\mult^{\follower_{i},k}_{1} = \mult^{\follower_{i},}_{2} = \hat{\mult}_{i}$ for~$i=1,2$, with the formula~\eqref{eq:exmp_mpcc_crcq_mult_form}, one obtain that
  \[
    \mult^{\follower_{i},k} \to \bar{\mult}_{i} := \left(\frac{1}{2},\frac{1}{2}\right) \in \Mult^{\follower_{i}}(\bar{\control}^{\leader},\bar{\control}^{\follower_{1}},\bar{\control}^{\follower_{2}}),
  \]
  following the expression~\eqref{eq:exmp_mpcc_crcq_Mult}.
  This actually proves that~$(\bar{\control}^{\leader},\bar{\control}^{\follower_{1}},\bar{\control}^{\follower_{2}},\bar{\mult}_{1},\bar{\mult}_{2})$, which is feasible for the MPCC~\eqref{eq:exmp_crcq_MPCC}, is not an optimal point for~\eqref{eq:exmp_crcq_MPCC}.
  However, the point~$(\bar{\control}^{\leader},\bar{\control}^{\follower_{1}},\bar{\control}^{\follower_{2}},\bar{\mult}^{\follower_{1}},\bar{\mult}^{\follower_{2}})$, with~$\bar{\mult}^{\follower_{i}} \in \Mult^{\follower_{i}}(\bar{\control}^{\leader},\bar{\control}^{\follower_{1}},\bar{\control}^{\follower_{2}}) \setminus \{\bar{\mult}_{i}\}$ for~$i=1,2$, is optimal for~\eqref{eq:exmp_crcq_MPCC}.

  We have thus shown that the MPCC reformulation of a SLMFG may have a local solution that is not (locally) optimal for the original SLMFG problem, as pointed out with the above procedure.
  This is because the followers' problems do not satisfy the CRCQ condition at~$(\bar{\control}^{\leader}, \bar{\control}^{\follower_{1}}, \bar{\control}^{\follower_{2}}) = (0,(0,0),(0,0))$.
  Indeed, it could be observed from the following arguments:
  \begin{itemize}[label=$\circ$, leftmargin=*]
    \item The set of active constraints for the follower~$\follower_{i}$~($i=1,2$) at~$(\bar{\control}^{\leader}, \bar{\control}^{\follower_{1}}, \bar{\control}^{\follower_{2}})$ is~$\Active^{\follower_{i}} = \{1,2\}$.
    \item The matrix
      \[
        (\grad_{\control^{\follower_{i}}} \confun_{j}^{\follower_{i}}(\bar{\control}^{\leader},\bar{\control}^{\follower_{i}}))_{j \in \Active^{\follower_{i}}} = \begin{pmatrix} \phantom{-}0 & \phantom{-}0\phantom{,} \\ -2 & -2\phantom{,} \end{pmatrix}
      \]
      has rank equals to~$1$.
    \item At any feasible point~$(\hat{\control}^{\leader},\hat{\control}^{\follower_{1}},\hat{\control}^{\follower_{2}})$ with~$\hat{\control}^{\leader} > 0$, the matrix
      \[
        (\grad_{\control^{\follower_{i}}} \confun^{\follower_{i}}_{j} (\hat{\control}^{\leader},\hat{\control}^{\follower_{i}}))_{j \in \Active^{\follower_{i}}} = \begin{pmatrix} 2(\hat{\control}^{\follower_{i}}_{1} - \hat{\control}^{\leader}) & 2(\hat{\control}^{\follower_{i}}_{1} + \hat{\control}^{\leader}) \\ 2(\hat{\control}^{\follower_{i}}_{2} - \hat{\control}^{\leader} - 1) & 2(\hat{\control}^{\follower_{i}}_{2} - \hat{\control}^{\leader} - 1) \end{pmatrix}
      \]
      has rank equals to~$2$.
  \end{itemize}
\end{example}

\section{Multiple followers with shared constraints}\label{sec:GNEP-lower}

The results presented in the previous section deal with SLMFGs where the followers solve a Nash equilibrium problem.
This feature means that the choice of each follower is constrained only to himself (and the single leader) and not the other followers.
This can be too restrictive in some practical situations, and deems that the choice of a follower could as well be constrained by the other followers, \emph{i.e.} to consider a \emph{generalized Nash equilibrium problem} (briefly \emph{GNEP}).

In this section, we partially answer such a call by showing that a GNEP with \emph{shared constraints} (in the spirit of~\citet{zbMATH03229265}) can be equivalently reformulated as a classical NEP under joint convexity assumption.
As a consequence, all the results presented in Section~\ref{sec:main} are then applicable to the SLMFG with Rosen's GNEP lower-level.
This situation is prevalent in a decision problem with a limited availability of resources, \emph{e.g.} followers making purchases on a limited amount of bananas and coconuts in a market to make poached banana in coconut milk.

To be accurate, let us fix, for each~$\control^{\leader} \in \Control^{\leader}$, a new constraint set~$\shared{\Constraint}^{\Follower}(\control^{\leader}) \subset \Control^{\Follower}$ with
\begin{subequations}\label{eq:follower_GNEP_shared_constraint}
  \begin{alignat}{3}
    & \shared{\Constraint}^{\Follower}(\control^{\leader}) = \{\control^{\Follower} \in \Control^{\Follower} \mid \shared{\confun}^{\Follower}_{j}(\control^{\leader},\control^{\Follower}) \leq 0 \quad &&(\forall j = 1,\dots,\shared{p}^{\Follower})\} \\
    & \shared{\confun}^{\Follower}_{j} : \Control^{\leader} \times \Control^{\Follower} \to \R &&(\forall j = 1,\dots,\shared{p}^{\Follower}).
  \end{alignat}
\end{subequations}
We also introduce the notation 
\[
  \shared{\Constraint}^{\follower}(\control^{\leader},\control^{-\follower}) := \{\control^{\follower} \in \Control^{\follower} \mid (\control^{\follower},\control^{-\follower}) \in \shared{\Constraint}^{\Follower}(\control^{\leader})\} \subset \Control^{\follower}
\]
for any~$\follower \in \Follower$,~$\control^{\leader} \in \Control^{\leader}$ and~$\control^{-\follower} \in \Control^{-\follower}$.
Additionally, we suppose for each $\follower \in \Follower$ that the objective function $\criterion^{\follower}$ is independent of $\control^{-\follower}$, \emph{i.e.}, $\criterion^{\follower} : \Control^{\leader} \times \Control^{\follower} \to \R$.
The followers' problems, given a signal $\control^{\leader} \in \Control^{\leader}$, becomes: Find a vector $\control^{\Follower} \in \Control^{\Follower}$ such that for each $\follower \in \Follower$, the element $\control^{\follower}$ solves the following optimization
\begin{subequations}\label{eq:follower_GNEP}
  \begin{empheq}[left=\empheqlbrace]{align}
    \min_{\hat{\control}^{\follower}} \quad & \criterion^{\follower}(\control^{\leader},\hat{\control}^{\follower}) \\
    \text{s.t.} \quad & \hat{\control}^{\follower} \in \Constraint^{\follower}(\control^{\leader}) \\
                      & \hat{\control}^{\follower} \in \shared{\Constraint}^{\follower}(\control^{\leader},\control^{-\follower}).
  \end{empheq}
\end{subequations}
This turns the followers' problems into a $\control^{\leader}$-parametrized \emph{Rosen's generalized Nash equilibrium problem} (briefly, \emph{RGNEP}), and we refer to the solution set of \eqref{eq:follower_GNEP} with $\RGNEP(\control^{\leader})$.

Now, let us show that lower-level RGNEP could be reformulated with a NEP.
\begin{theorem}\label{thm:GNEP-Opt-Rosen}
  The RGNEP given in \eqref{eq:follower_GNEP} can be reformulated into an optimization problem with the same global solution set, provided that $\bar{\confun}^{\Follower}_{j}$ is jointly convex.
\end{theorem}
\begin{proof}
  Let~$\control^{\leader} \in \Control^{\leader}$ be given.
  We define a function~$\tilde{\criterion}^{\Follower} : \Control^{\leader} \times \Control^{\Follower} \to \R$ and a constraint set~$\Constraint^{\Follower} \subset \Control^{\Follower}$ by
  \begin{align*}
    & \tilde{\criterion}^{\Follower} (\control^{\leader},\control^{\Follower}) := \sum_{\follower \in \Follower} \criterion^{\follower}(\control^{\leader}, \control^{\follower}) \\
    & \Constraint^{\Follower}(\control^{\leader}) := \prod_{\follower \in \Follower} \Constraint^{\follower}(\control^{\leader}).
  \end{align*}
  We now consider the following~$\control^{\leader}$-parametrized optimization problem
  \begin{subequations}
    \begin{empheq}[left=\empheqlbrace]{align}
      \min_{\control^{\Follower}} \quad & \tilde{\criterion}^{\Follower}(\control^{\leader},\control^{\Follower}) \\
      \text{s.t.}\quad & \control^{\Follower} \in \Constraint^{\Follower}(\control^{\leader}) \\
                       & \control^{\Follower} \in \shared{\Constraint}^{\Follower}(\control^{\leader}),
    \end{empheq}
  \end{subequations}
  whose (global) solution set will be denoted with~$\widetilde{\Opt}(\control^{\leader})$.
  We would like to prove that~$\RGNEP(\control^{\leader}) = \widetilde{\Opt}(\control^{\leader})$.

  It is clear that~$\widetilde{\Opt}(\control^{\leader}) \subset \RGNEP(\control^{\leader})$, hence we focus on showing~$\RGNEP(\control^{\leader}) \subset \widetilde{\Opt}(\control^{\leader})$.
  Suppose that~$\control^{\Follower} \in \RGNEP(\control^{\leader})$.
  This is equivalent to saying that for any~$\follower \in \Follower$, there exists~$g^{\follower} \in \partial_{\control^{\follower}} \criterion^{\follower} (\control^{\leader},\control^{\follower})$ such that
  \[
    \langle g^{\follower} , \hat{\control}^{\follower} - \control^{\follower} \rangle \geq 0
  \]
  for every~$\hat{\control}^{\follower} \in \Constraint^{\follower}(\control^{\leader}) \cap \shared{\Constraint}^{\follower}(\control^{\leader},\control^{-\follower})$.
  This is the same as saying
  \[
    \langle \tilde{g}^{\Follower}, \hat{\control}^{\Follower} - \control^{\Follower} \rangle \geq 0
  \]
  for every~$\hat{\control}^{\Follower} \in \Constraint^{\Follower}(\control^{\leader}) \cap \shared{\Constraint}^{\Follower}(\control^{\leader})$, where
  \[
    \tilde{g}^{\Follower} = (g^{\follower_{1}},\dots,g^{\follower_{n}}) \in \prod_{\follower \in \Follower} \partial_{\control^{\follower}} \criterion^{\follower}(\control^{\leader},\control^{\follower}) = \prod_{\follower \in \Follower} \partial_{\control^{\follower}} \tilde{\criterion}^{\Follower}(\control^{\leader},\control^{\follower}) = \partial_{\control^{\Follower}} \tilde{\criterion}^{\Follower}(\control^{\leader},\control^{\Follower}).
  \]
  We have thus proved that~$\control^{\Follower} \in \widetilde{\Opt}(\control^{\leader})$.
\end{proof}

The above technique could also be adapted to a more general GNEP in which the followers are partitioned into groups, and players within each group have shared constraints only among them.
Let us develop some additional notations here. For any subset~$\Bgent \subset \Follower$, we write~$\Control^{\Bgent} = \prod_{\bgent \in \Bgent} \Control^{\bgent}$ and its generic element~$\control^{\Bgent} = (\control^{\bgent})_{\bgent \in \Bgent}$.
We also write~$\Control^{-\Bgent} = \prod_{\agent \in \Follower\setminus\Bgent}$ and its generic element~$\control^{-\Bgent} = (\control^{\agent})_{\agent \in \Follower\setminus\Bgent}$.
Any vector~$\control^{\Follower}$ can then be represented with~$\control^{\Follower} = (\control^{\Bgent},\control^{-\Bgent})$.

Formally, suppose that the set~$\Follower$ of followers is partitioned into~$\Follower_{1},\dots,\Follower_{m}$, each of them are mutually disjoint.
For each~$i=1,\dots,m$ and each~$\control^{\leader}$, we consider the following shared constraint set
\begin{subequations}
  \begin{alignat}{2}
    & \bar{\Constraint}^{\Follower_{i}}(\control^{\leader}) = \{ \control^{\Follower_{i}} = (\control^{\follower})_{\follower \in \Follower_{i}} \mid \bar{\confun}^{\Follower_{i}}_{j}(\control^{\leader},\control^{\Follower_{i}}) \leq 0 \quad && (\forall j = 1,\dots,\bar{p}^{\Follower_{i}})\} \\
    & \bar{\confun}^{\Follower_{i}}_{j} : \Control^{\leader} \times \Control^{\Follower_{i}} \to \R && (\forall j=1,\dots,\bar{p}^{\Follower_{i}}).
  \end{alignat}
\end{subequations}
We focus on the situation that a follower's objective function is independent of the decision of the one within the same group.
This means, for each~$\follower \in \Follower_{i}$ ($i=1,\dots,m$), the objective is a real function~$\criterion^{\follower}$ is defined on~$\Control^{\leader} \times \Control^{\follower} \times \Control^{-\Follower_{i}} = \Control^{\leader} \times \Control^{\follower} \times \prod_{\gollower \in \Follower\setminus\Follower_{i}} \Control^{\gollower}$.
The projection of the above constraint system to each of the follower~$\follower \in \Follower_{i}$ is then defined as
\[
  \bar{\Constraint}^{\follower}(\control^{\leader},\control^{\Follower_{i}\setminus\{\follower\}}) = \{ \control^{\follower} \in \Control^{\follower} \mid (\control^{\follower},\control^{\Follower_{i}\setminus\{\follower\}}) \in \bar{\Constraint}^{\Follower_{i}} \}.
\]
For each~$\control^{\leader}$, the corresponding RGNEP of all the followers then have the following form: Find~$\control^{\Follower} \in \Control^{\Follower}$ such that for each~$i=1,\dots,m$ and each~$\follower \in \Follower_{i}$, the component~$\control^{\follower}$ solves
\begin{subequations}\label{eq:GNEP-lower-full}
  \begin{empheq}[left=\empheqlbrace]{align}
    \min_{\hat{\control}^{\follower}} \quad & \criterion^{\follower}(\control^{\leader},\hat{\control}^{\follower},\control^{-\Follower_{i}}) \\
    \text{s.t.} \quad & \hat{\control}^{\follower} \in \Constraint^{\follower}(\control^{\leader}) \\
                      & \hat{\control}^{\follower} \in \bar{\Constraint}^{\follower} (\control^{\leader},\control^{\Follower_{i}\setminus\{\follower\}})
  \end{empheq}
\end{subequations}

With the same proof with that of Theorem~\ref{thm:GNEP-Opt-Rosen} applied to each follower group~$\Follower_{i}$, we could obtain the following, more general consequence.
\begin{theorem}\label{thm:GNEP-NEP-Rosen}
  The RGNEP given in~\eqref{eq:GNEP-lower-full} can be reformulated into a NEP with the same global solution set, provided that each~$\bar{\confun}^{\Follower_{i}}_{j}$ is jointly convex.
\end{theorem}
\begin{proof}
  Consider the leader~$\leader$ and a new set of followers~$\tilde{\Follower} := \{\Follower_{1},\dots,\Follower_{m}\}$.
  With this revised set of players, consider for each given~$\control^{\leader} \in \Control^{\leader}$ the following NEP: Find~$\control^{\tilde{\Follower}} = (\control^{\Follower_{1}},\dots,\control^{\Follower_{m}}) \in \prod_{i=1}^{m} \Control^{\tilde{\Follower}_{i}}$ in which for each~$i = 1,\dots,m$, the component~$\control^{\Follower_{i}}$ solves the problem
  \begin{subequations}\label{eq:NEP_reformulation_of_GNEP}
    \begin{empheq}[left=\empheqlbrace]{align}
      \min_{\control^{\Follower_{i}}} \quad & \tilde{\criterion}^{\Follower_{i}}(\control^{\leader},\control^{\Follower_{i}},\control^{-\Follower_{i}}) = \sum_{\follower \in \Follower_{i}} \criterion^{\follower}(\control^{\leader},\control^{\follower},\control^{-\Follower_{i}}) \\
      \text{s.t.} \quad & \control^{\Follower_{i}} \in \tilde{\Constraint}^{\Follower_{i}}(\control^{\leader}) = \prod_{\follower \in \Follower_{i}} \Constraint^{\follower}(\control^{\leader}) \\
                        & \control^{\Follower_{i}} \in \shared{\Constraint}^{\Follower_{i}}(\control^{\leader}).
    \end{empheq}
  \end{subequations}
  With a similar proof to that of Theorem~\ref{thm:GNEP-Opt-Rosen}, we could show that ~\eqref{eq:NEP_reformulation_of_GNEP} is equivalent to the RGNEP~\eqref{eq:GNEP-lower-full}.
\end{proof}

Now, we consider the following SLMFG, where the lower-level problem is a RGNEP:
\begin{subequations}\label{eq:SLMFG-GNEP-Lower}
  \begin{empheq}[left=\empheqlbrace]{align}
    \min_{\control^{\leader},\control^{\Follower}} \quad & \criterion^{\leader}(\control^{\leader}, \control^{\Follower}) \\
    \text{s.t.} \quad & \control^{\leader} \in \Constraint^{\leader} \\
                      & \control^{\Follower} \in \RGNEP(\control^{\leader}),
  \end{empheq}
\end{subequations}
where~$\RGNEP(\control^{\leader})$ now refers to the problem~\eqref{eq:GNEP-lower-full}.
If each function~$\confun^{\Follower_{i}}_{j}$ is jointly convex, then the above SLMFG~\eqref{eq:SLMFG-GNEP-Lower} reduces to a simplified SLMFG of the form

\begin{subequations}\label{eq:SLMFG-GNEP-Lower-reduced}
  \begin{empheq}[left=\empheqlbrace]{align}
    \min_{\control^{\leader},\control^{\tilde{\Follower}}} \quad & \criterion^{\leader}(\control^{\leader}, \control^{\tilde{\Follower}}) \\
    \text{s.t.} \quad & \control^{\leader} \in \Constraint^{\leader} \\
                      & \control^{\tilde{\Follower}} \in \tilde{NEP}(\control^{\leader}),
  \end{empheq}
\end{subequations}
where~$\tilde{\Follower} = \{\Follower_{1},\cdots,\Follower_{m}\}$ and~$\tilde{\NEP}(\control^{\leader})$ is the solution set of the problem~\eqref{eq:NEP_reformulation_of_GNEP}.
Since the above reduced problem~\eqref{eq:SLMFG-GNEP-Lower-reduced} is explicit in terms of the constraint functions~$\confun^{\follower}_{j}$'s and~$\bar{\confun}^{\Follower_{i}}_{j}$'s, the results of Section~\ref{sec:main} are applicable to the SLMFG with group-shared constraints~\eqref{eq:SLMFG-GNEP-Lower}, through its reduced form~\eqref{eq:SLMFG-GNEP-Lower-reduced}, with assumptions directly verifiable on these constraint functions.

\section*{Ackowledgments}

The first author (Parin Chaipunya) is grateful to the Center of Excellence in Theoretical and Computational Science (TaCS-CoE), at King Mongkut's University of Technology Thonburi, for its financial support.


\end{document}